\documentclass[10pt,a4paper]{article}
\usepackage{amsmath}
\usepackage{amsthm}
\usepackage{amsfonts}
\usepackage{latexsym}

\newcounter{alphthm}
\newtheorem{thm}{Theorem}[section]
\newtheorem{lem}[thm]{Lemma}
\newtheorem{cor}{Corollary}[section]

\theoremstyle{definition}

\newcommand{\be}{\begin{equation}}
\newcommand{\ee}{\end{equation}}
\newcommand{\pa}{{\partial}}

\newcommand{\g}{{\bf g}}
\newcommand{\pxi}{{\pa \over \pa x^i}}

\title{\bf Generalized P-Reducible Finsler Metrics}
\author{E. Peyghan,  A. Tayebi and A. Heydari}

\begin{document}

\maketitle

\begin{abstract}
In this paper, we study a class of Finsler metrics which contains the class of P-reducible metrics. Finsler metrics in this class are called generalized P-reducible metrics. We consider generalized P-reducible metrics with  scalar flag curvature and find a condition under which these metrics reduce to C-reducible  metrics. This  generalize  Matsumoto's theorem, which describes  the equivalency  of C-reducibility  and P-reducibility on Finsler manifolds with  scalar curvature. Then we show that generalized P-reducible metrics with vanishing stretch curvature are C-reducible.
\end{abstract}

\bigskip

\section{Introduction}
In Finsler geometry,  there are several important non-Riemannian  quantities:  the  Cartan torsion ${\bf C}$,  the Berwald curvature ${\bf B}$, the Landsberg curvature ${\bf L}$, the  mean Landsberg curvature ${\bf J}$ and the stretch curvature ${\bf \Sigma}$, etc \cite{NST}.  They all vanish  for Riemannian metrics, hence they are said to be   non-Riemannian.  The study of these quantities is benefit for us to make out their distinction and the nature of Finsler geometry.

Let $F$ be a Finsler metric on a manifold $M$. The geodesics of $F$ are characterized locally by the equation
$ \ddot c^i+2G^i(\dot c)=0$, where $G^i = {1\over 4} g^{ik} \{ 2 {\pa g_{pk}\over \pa x^q}-{\pa g_{pq}\over \pa x^k} \} y^py^q$. $G^i$ are coefficients of a spray defined on $M$ denoted by ${\bf G}(x,y)=y^i\frac{\partial}{\partial x^i}-2G^i\frac{\partial}{\partial y^i}$.  A Finsler metric $F$ is  called a Berwald metric if
$G^i =\frac {1}{2} \Gamma^i_{jk}(x)y^jy^k$ are quadratic in $y\in T_xM$  for any $x\in M$ \cite{TP1}.

There is another class of Finsler metrics which contain the class of Berwald metrics as a special case, namely the class of Landsberg metrics. For introducing the Landsberg curvature,  let us  describe some non-Riemannian  quantities in Finsler geometry.  The second derivatives of ${1\over 2} F_x^2$ at $y\in T_xM_0$ is an inner product $\g_y$ on $T_xM$.  The third order derivatives of ${1\over 2} F_x^2$ at  $y\in T_xM_0$ is a symmetric trilinear forms ${\bf C}_y$ on $T_xM$. We call $\g_y$ and ${\bf C}_y$ the  fundamental form and  the Cartan torsion, respectively. The rate of change of the Cartan torsion along geodesics is the  Landsberg curvature  $ {\bf L}_y$ on $T_xM$ for any $y\in T_xM_0$. Set ${\bf J}_y:= \sum_{i=1}^n {\bf L}_y(e_i, e_i, \cdot )$, where $\{e_i\}$ is an orthonormal basis for $(T_xM, \g_y)$. ${\bf J}_y$ is called the  mean Landsberg curvature. $F$ is said to be  Landsbergian if ${\bf L}=0$,  and  weakly Landsbergian if ${\bf J}=0$ \cite{Sh5}.

On the other hand, various interesting special forms of Cartan, Berwald  and Landsberg tensors have been obtained by some Finslerians. The Finsler spaces having such special forms have been called C-reducible, P-reducible, general relatively isotropic Landsberg, and etc. In \cite{M3}, Matsumoto  introduced the notion of C-reducible Finsler metrics and proved that any Randers metric is C-reducible. Later on, Matsumoto-H\={o}j\={o} proves that the converse is true too \cite{MaHo}. A Randers metric $F=\alpha+\beta$ is just a Riemannian metric $\alpha$ perturbated by a one form $\beta$. Randers metrics have important applications both in mathematics and physics.

As a generalization of C-reducible metrics, Matsumoto-Shimada introduced the notion of P-reducible metrics \cite{MS}. This class of Finsler metrics have some interesting physical means and contains Randers metrics as a special case.

In \cite{Pr}, B. N. Prasad introduced a new class of Finsler spaces which contains the notion of P-reducible, as  special case. A Finsler metric $F$ is called generalized P-reducible if its Landsberg curvature is given by following
\[
L_{ijk}=\lambda C_{ijk}+ a_ih_{jk}+a_jh_{ki}+a_kh_{ij},
\]
where $\lambda=\lambda(x,y)$  is a  scalar function on $TM$, $a_i=a_i(x)$ is scalar function on $M$ and $h_{ij}=g_{ij}-F^{-2}y_iy_j$ is the angular metric. $\lambda$ and $a_i$ are homogeneous function of degree 1  and degree 0 with respect to $y$, respectively. By definition, we have $a_iy^i=0$.
If $a_i=0$, then $F$ is reduce to general isotropic Landsberg metric and if $\lambda=0$ then $F$ is a P-reducible metric. Therefore the study of this class of Finsler spaces will enhance our understanding of the geometric meaning of Randers metrics.

For a Finsler manifold $(M, F)$, the flag curvature is  a function ${\bf K}(P, y)$ of tangent planes $P\subset T_xM$ and  directions $y\in P$. $F$  is said to be  of scalar flag curvature if the flag curvature ${\bf K}(P, y)={\bf K}(x, y)$ is independent of flags $P$ associated with any fixed flagpole $y$.  One of the important problems in Finsler geometry is to characterize Finsler manifolds of scalar flag curvature \cite{NST}\cite{NT}.

In \cite{M3}, Matsumoto proved that every P-reducible Finsler metric of non-zero scalar flag curvature reduces to a C-reducible Finsler metric. In this paper, we study the class generalized P-reducible Finsler metrics with  scalar flag curvature and find a condition under which this metrics reduce to C-reducible metric.  More precisely, we prove the following.
\begin{thm}\label{MainTheorem1} Let $(M, F)$ be a  generalized P-reducible Finsler manifold of non-zero  scalar flag curvature ${\bf K}$. Suppose that $F$ satisfy in $\lambda'+\lambda^2+{\bf K}F^2\neq 0$, where $\lambda'=\lambda_{|l}y^l$. Then $F$ is a C-reducible  metric.
\end{thm}

Let $(M, F)$ be a  generalized P-reducible Finsler manifold with vanishing Landsberg curvature. Then $F$ reduces to a C-reducible metric. On the other hand, every C-reducible Landsberg metric is a Berwald metric. Therefore every generalized P-reducible  manifold with vanishing Landsberg curvature is a Berwald manifold.

As a generalization of Landsberg curvature, L. Berwald introduced a non-Riemannian curvature so-called stretch curvature and denoted by ${\bf \Sigma}_y$ \cite{Be1}. He  showed that this tensor vanishes if and only if the length of a vector remains unchanged under the parallel displacement along an infinitesimal parallelogram. Then, this curvature investigated by Shibata  and Matsumoto \cite{M2}.  A Finsler metric  is said to be  stretch metric if ${\bf \Sigma}=0$. In this paper, we consider generalized P-reducible  manifolds with vanishing stretch  curvature and prove the following.

\begin{thm}\label{MainTheorem2} Let $(M, F)$ be a  generalized P-reducible Finsler manifold. Suppose that $F$  is a stretch metric. Then $F$ is a C-reducible  metric.
\end{thm}

It is remarkable that, M. Matsumoto  proved that on a Finsler manifold $M$ with dimension $n\geq 3$, every positive definite C-reducible Finsler metric is a Randers metric \cite{M3}. Then by Theorem \ref{MainTheorem2}, we have the following.

\begin{cor}
Let $(M, F)$ be a  generalized P-reducible manifold with dimension  $n\geq 3$. Suppose that $F$  is a stretch metric. Then $F$ is a Randers  metric.
\end{cor}

There are many connections in Finsler geometry \cite{BT}\cite{TAE}\cite{TN}.  Throughout  this paper, we use  the Berwald connection on Finsler manifolds. The $h$- and $v$- covariant derivatives of a Finsler tensor field are denoted by `` $|$ " and ``, " respectively.

\section{Preliminaries}
Let $M$ be a n-dimensional $ C^\infty$ manifold. Denote by $T_x M $ the tangent space at $x \in M$,  by $TM=\cup _{x \in M} T_x M $ the tangent bundle of $M$ and by $TM_{0} = TM \setminus \{ 0 \}$  the slit tangent bundle of $M$.

A  Finsler metric on $M$ is a function $ F:TM \rightarrow [0,\infty)$ which has the following properties:\\
(i)\ $F$ is $C^\infty$ on $TM_{0}$;\\
(ii)\ $F$ is positively 1-homogeneous on the fibers of tangent bundle $TM$;\\
(iii)\ for each $y\in T_xM$, the following quadratic form ${\bf g}_y$ on
$T_xM$  is positive definite,
\[
{\bf g}_{y}(u,v):={1 \over 2} \frac{d^2}{dtds} \left[  F^2 (y+su+tv)\right]|_{s,t=0}, \ \
u,v\in T_xM.
\]
Let  $x\in M$ and $F_x:=F|_{T_xM}$.  To measure the
non-Euclidean feature of $F_x$, define ${\bf C}_y:T_xM\otimes T_xM\otimes
T_xM\rightarrow \mathbb{R}$ by
\[
{\bf C}_{y}(u,v,w):={1 \over 2} \frac{d}{dt}\left[{\bf g}_{y+tw}(u,v)
\right]|_{t=0}, \ \ u,v,w\in T_xM.
\]
The family ${\bf C}:=\{{\bf C}_y\}_{y\in TM_0}$  is called the Cartan torsion. It is well known that ${\bf{C}}=0$ if and only if $F$ is Riemannian \cite{Sh5}.

For $y\in T_x M_0$, define  mean Cartan torsion ${\bf I}_y$ by ${\bf I}_y(u):=I_i(y)u^i$, where
\[
I_i:=g^{jk}C_{ijk},
\]
$g^{jk}$ is the inverse of $g_{jk}$ and $u=u^i\frac{\partial}{\partial x^i}|_x$. By Deicke  Theorem, $F$ is Riemannian  if and only if ${\bf I}_y=0$.

\bigskip

Given a Finsler manifold $(M,F)$, then a global vector field ${\bf G}$ is
induced by $F$ on $TM_0$, which in a standard coordinate $(x^i,y^i)$
for $TM_0$ is given by ${\bf G}=y^i {{\partial} \over {\partial x^i}}-2G^i(x,y){{\partial} \over
{\partial y^i}}$, where
\[
G^i(x,y):={1\over 4}g^{il}(x, y)\Big\{\frac{\pa ^2 F^2}{\pa x^k \pa y^l}y^k-\frac{\pa F^2}{\pa x^l}\Big\}
\]
are called the spray coefficients of ${\bf G}$ \cite{TP1}.  Then {\bf G} is called the  spray associated  to $(M,F)$.  In local coordinates, a curve $c(t)$ is a geodesic if and only if its coordinates $(c^i(t))$ satisfy
\[
{d^2 x^i\over dt^2} + 2 G^i( x, {dx\over dt})=0.
\]

\bigskip

Let $(M, F)$ be a Finsler manifold. For   $y \in T_xM_0$, define the  Matsumoto torsion ${\bf M}_y:T_xM\otimes T_xM \otimes T_xM \rightarrow \mathbb{R}$ by ${\bf M}_y(u,v,w):=M_{ijk}(y)u^iv^jw^k$ where
\[
M_{ijk}:=C_{ijk} - {1\over n+1}  \{ I_i h_{jk} + I_j h_{ik} + I_k h_{ij} \},\label{Matsumoto}
\]
$h_{ij}:=FF_{y^iy^j}=g_{ij}-\frac{1}{F^2}g_{ip}y^pg_{jq}y^q$ is the angular metric. A Finsler metric $F$ is said to be C-reducible metric if ${\bf M}_y=0$.\\

The Finsler metric $F=\alpha+\beta$ is called a Randers metric, where $\alpha=\sqrt{a_{ij}y^iy^j}$ is a Riemannian metric, and $\beta=b_i(x)y^i$ be a 1-form on $M$ with $||\beta||_{\alpha}<1$. These metrics have important applications both in mathematics and physics.
\begin{lem}\label{MaHo}{\rm (\cite{Ma1}\cite{MaHo})}
\emph{A positive-definite Finsler metric $F$ on a manifold of dimension $n\geq 3$ is a Randers metric if and only if\  ${\bf M}_y =0$, $\forall y\in TM_0$.}
\end{lem}

\bigskip

The horizontal covariant derivatives of Cartan tensor ${\bf C}$ along geodesics give rise to  the  Landsberg curvature  ${\bf L}_y:T_xM\otimes T_xM\otimes T_xM\rightarrow \mathbb{R}$  defined by ${\bf L}_y(u,v,w):=L_{ijk}(y)u^iv^jw^k,
$, where
\[
L_{ijk}:=C_{ijk|s}y^s,
\]
$u=u^i{{\partial } \over {\partial x^i}}|_x$,  $v=v^i{{\partial }\over {\partial x^i}}|_x$ and $w=w^i{{\partial }\over {\partial x^i}}|_x$. The family ${{\bf L}}:=\{{{\bf L}}_y\}_{y\in TM_{0}}$  is called the Landsberg curvature. It is easy to see that, for a Finsler metric $F$ on a manifold $M$, the Minkwoski norm $F_x:=F|_{T_xM}$ on $T_xM$, for each non-zero tangent vector $y$ at $x$ induces a Riemannian metric $\hat{g}_x:=g_{ij}(x, y)dy^i\otimes dy^j$ on $T_xM_0=T_xM-\{0\}$. A Finsler metric  $F$ is called a Landsberg metric  if ${\bf L}=0$. It is well known that on a Landsberg manifold $(M,F)$, all $(T_xM_0, \hat{g}_x)$ are isometric as Banach spaces \cite{Sh5}.

\bigskip

The horizontal covariant derivatives of ${\bf I}$ along geodesics give rise to  the mean Landsberg curvature ${\bf J}_y(u): = J_i (y)u^i$, where
\[
J_i: = g^{jk}L_{ijk}.
\]
A Finsler metric $F$ is said to be weakly Landsbergian if ${\bf J}=0$ \cite{TP2}.

\bigskip

Define  ${\bf\bar M}_y:T_xM\otimes T_xM \otimes T_xM \rightarrow \mathbb{R}$ by ${\bf \bar M}_y(u,v,w):={\bar M}_{ijk}(y)u^iv^jw^k$ where
\[
{\bar M}_{ijk}:=L_{ijk} - {1\over n+1}  \{ J_i h_{jk} + J_j h_{ik} + J_k h_{ij} \}.
\]
A Finsler metric $F$ is said to be P-reducible if   ${\bf\bar M}_y=0$. The notion of P-reducibility was given by Matsumoto-Shimada \cite{MS}. It is obvious that every C-reducible metric is a P-reducible metric.

\bigskip

Define the stretch curvature ${\bf \Sigma}_y:T_xM\otimes T_xM \otimes T_xM  \otimes T_xM\rightarrow \mathbb{R}$ by ${\bf \Sigma}_y(u, v, w,z):={\Sigma}_{ijkl}(y)u^iv^jw^kz^l$, where
\[
{\Sigma}_{ijkl}:=2(L_{ijk|l}-L_{ijl|k}).
\]
A Finsler metric  is said to be   stretch metric if ${\bf \Sigma}=0$. The notion of stretch curvature was introduced by L. Berwald   a generalization of Landsberg curvature \cite{Be1}. He proved that this stretch curvature of a Finsler manifold $M$ vanishes if and only if the length of a vector remains unchanged under the parallel displacement along an infinitesimal parallelogram. It is easy to see that, every Landsberg metric is a  stretch metric.

\bigskip

The Riemann curvature ${\bf K}_y= K^i_{\ k}  dx^k \otimes \pxi|_x : T_xM \to T_xM$ is a family of linear maps on tangent spaces,
defined by
\[
K^i_{\ k} = 2 {\pa G^i\over \pa x^k}-y^j{\pa^2 G^i\over \pa x^j\pa y^k}
+2G^j {\pa^2 G^i \over \pa y^j \pa y^k} - {\pa G^i \over \pa y^j}{\pa G^j \over \pa y^k}.  \label{Riemann}
\]
For a flag $P={\rm span}\{y, u\} \subset T_xM$ with flagpole $y$, the flag curvature ${\bf K}={\bf K}(P, y)$ is defined by
\[
{\bf K}(P, y):= {\g_y (u, {\bf K}_y(u))
\over \g_y(y, y) \g_y(u,u)
-\g_y(y, u)^2 }.
\]
We say that  a Finsler metric $F$ is   of scalar curvature if for any $y\in T_xM$, the flag curvature ${\bf K}= {\bf K}(x, y)$ is a scalar function on the slit tangent bundle $TM_0$.

\section{Proof of Theorem \ref{MainTheorem1}}

In this section, we are going to prove the Theorem \ref{MainTheorem1}. To prove it, we need the following.
\begin{lem}
Let $(M, F)$ be a generalized P-reducible Finsler manifold. Then the Matsumoto torsion of $F$ satisfy in following
\begin{equation} \label{P7}
M_{ijk|s}y^s=\lambda(x,y)M_{ijk}.
\end{equation}
\end{lem}
\begin{proof}  Let $F$ be a generalized P-reducible metric
\begin{equation} \label{P2}
L_{ijk}=\lambda C_{ijk}+ a_ih_{jk}+a_jh_{ki}+a_kh_{ij}.
\end{equation}
Contracting (\ref{P2}) with $g^{ij}$  and using the relations
\[
g^{ij}h_{ij}=n-1\ \  \textrm{and} \ \ g^{ij}(a_ih_{jk})=g^{ij}(a_jh_{ik})=a_k
\]
implies that
\begin{equation} \label{P3}
J_k=\lambda I_k+ (n+1)a_k.
\end{equation}
Then
\begin{equation} \label{P4}
a_i=\frac{1}{n+1}J_i-\frac{\lambda}{n+1} I_i.
\end{equation}
Putting (\ref{P4}) in (\ref{P2}) yields
\begin{eqnarray}
\nonumber L_{ijk}= \lambda C_{ijk}\!\!\!\!&+&\!\!\!\!\ \frac{1}{n+1}\{J_ih_{jk}+J_jh_{ki}+J_kh_{ij}\}\\ \!\!\!\!&-&\!\!\!\!\ \frac{\lambda}{n+1}\{I_ih_{jk}+I_jh_{ki}+I_kh_{ij}\}.\label{P5}
\end{eqnarray}
By simplifying  (\ref{P5}), we get
\begin{equation}
L_{ijk}-\frac{1}{n+1}(J_ih_{jk}+J_jh_{ki}+J_kh_{ij})= \lambda \{C_{ijk}- \frac{1}{n+1}(I_ih_{jk}+I_jh_{ki}+I_kh_{ij})\}.\label{P6}
\end{equation}
The equation   (\ref{P6}) is equivalent to (\ref{P7}).
\end{proof}

\bigskip

\begin{lem}{\rm (\cite{Mo1})} Landsberg curvature and Riemann Curvature are related by the following equation
\begin{eqnarray}
L_{ijk|m}y^m + C_{ijm}R^m_{\ \ k} & = &  - {1\over 3}g_{im}R^m_{\ \ k\cdot j}
- {1\over 3} g_{jm} R^m_{\ \ k\cdot i}\nonumber\\
& &  - {1\over 6} g_{im}R^m_{\ \ j\cdot k}
- {1\over 6} g_{jm}R^m_{\ \ i\cdot k}. \label{Moeq1}
\end{eqnarray}
Contracting (\ref{Moeq1}) with $g^{ij}$ gives
\begin{equation}
J_{k|m}y^m + I_mR^m_{\  k}  = -  {1\over 3}\Big \{ 2 K^m_{\   k\cdot m} +  K^m_{\   m\cdot k}\Big \}. \label{Moeq2}
\end{equation}
\end{lem}

\bigskip

\noindent {\bf Proof of Theorem \ref{MainTheorem1}}:  Let $F$ be a generalized P-reducible Finsler metric  of scalar curvature ${\bf K}= {\bf K}(x, y)$. This is equivalent to the following identity:
\be
R^i_{\ k} = {\bf K} F^2 \; h^i_k, \label{Kikiso1}
\ee
where $h^i_k := g^{ij} h_{jk}$.
Differentiating (\ref{Kikiso1}) yields
\be
R^i_{\ k\cdot l}
=  {\bf K}_{\cdot l} F^2 \; h^i_k + {\bf K} \Big \{ 2 g_{lp}y^p \delta^i_k - g_{kp}y^p  \delta^i_l- g_{kl} y^i  \Big \}.\label{MKdiff}\ee
By  (\ref{Moeq1}),  (\ref{Moeq2}) and (\ref{MKdiff}), we obtain
\be
 L_{ijk|m}y^m =  - {1\over 3}F^2 \Big \{ {\bf K}_{\cdot i} h_{jk}   + {\bf K}_{\cdot j} h_{ik}  + {\bf K}_{\cdot k} h_{ij} + 3 {\bf K} C_{ijk}\Big \} \label{AZeq1}
\ee
and
\be
J_{k|m}y^m = - {1\over 3}F^2\Big \{ (n+1) {\bf K}_{\cdot k} + 3{\bf K} I_k \Big \}.\label{AZeq2}
\ee
On the other hand, we have

\be
L_{ijk} = C_{ijk|m}y^m, \ \ \ \ \ \ J_i = I_{i|m}y^m.\label{LJE}
\ee
By (\ref{LJE}), we get
\[ C_{ijk|p|q}y^py^q = L_{ijk|m}y^m, \ \ \ \ \ I_{k|p|q}
y^py^q = J_{k|m}y^m.
\]
Then we get
\be
M_{ijk|p|q} = L_{ijk|m}y^m - {1\over n+1} \Big \{ J_{i|m}y^m h_{jk}
+ J_{j|m}y^m h_{ik} + J_{k|m}y^m h_{ij} \Big \}.
\label{AZeq3}
\ee
Plugging  (\ref{AZeq1}) and (\ref{AZeq2}) into (\ref{AZeq3}) yields
\be
 M_{ijk|p|q} y^p y^q + {\bf K} F^2 M_{ijk} =0. \label{Sijk}
\ee
By (\ref{P7}) we have
\begin{eqnarray} \label{P9}
\nonumber M_{ijk|p|q} y^p y^q \!\!\!\!&=&\!\!\!\!\ \lambda'(t)  M_{ijk}+\lambda(t)  M_{ijk|p} y^p\\
\!\!\!\!&=&\!\!\!\!\ (\lambda'+\lambda^2) M_{ijk}.
\end{eqnarray}
By (\ref{Sijk}) and (\ref{P9}) we get
\begin{equation} \label{P10}
(\lambda'+\lambda^2+{\bf K}F^2) M_{ijk}=0.
\end{equation}
By assumption $\lambda'+\lambda^2+{\bf K}F^2\neq 0$, then $ M_{ijk}=0$. Then by definition, $F$ is a C-reducible metric.
\qed

\bigskip

By Theorem \ref{MainTheorem1}, we have the following.

\begin{cor}
Let $(M, F)$ be a Finsler metric of non-zero scalar flag curvature. Suppose that $F$ be a P-reducible metric. Then $F$ is  a C-reducible metric.
\end{cor}
\section{Proof of Theorem \ref{MainTheorem2}}
In this section, we are going to prove the Theorem \ref{MainTheorem2}.
\bigskip

\noindent {\bf Proof of Theorem \ref{MainTheorem2}}: Let $F$ be a stretch metric
\be
L_{ijk|l}-L_{ijl|k}=0.\label{S1}
\ee
Contracting (\ref{S1}) with $y^l$ yields
\be
L_{ijk|l}y^l=0.\label{S2}
\ee
By definition, we have
\begin{equation}
L_{ijk}=\lambda C_{ijk}+ a_ih_{jk}+a_jh_{ki}+a_kh_{ij}.\label{S3}
\end{equation}
Taking a horizontal derivation of (\ref{S3}) implies that
\be
L_{ijk|l}y^l=\lambda' C_{ijk}+\lambda L_{ijk}+ a'_ih_{jk}+a'_jh_{ki}+a'_kh_{ij},\label{S4}
\ee
where  $a'_i=a_{i|l}y^l$. By putting  (\ref{S3}) in  (\ref{S4}), it follows that
\be
L_{ijk|l}y^l=(\lambda'+\lambda^2) C_{ijk}+(\lambda a_i+ a'_i)h_{jk}+(\lambda a_j+a'_j)h_{ki}+(\lambda a_k+a'_k)h_{ij}.\label{S5}
\ee
By (\ref{S2}) and   (\ref{S5}), we have
\be
 C_{ijk}=\frac{-1}{\lambda'+\lambda^2}\Big[(\lambda a_i+ a'_i)h_{jk}+(\lambda a_j+a'_j)h_{ki}+(\lambda a_k+a'_k)h_{ij}\Big].\label{S6}
\ee
Multiplying   (\ref{S6}) with $g^{ij}$ yields
\be
I_k=\frac{-(n+1)}{\lambda'+\lambda^2}(\lambda a_k+a'_k),\label{S7}
\ee
or equivalently
\be
\lambda a_k+a'_k=-\frac{\lambda'+\lambda^2}{n+1} I_k.\label{S8}
\ee
By plugging   (\ref{S8}) in   (\ref{S6}), we have
\be
C_{ijk}=\frac{1}{n+1}(I_ih_{jk}+I_jh_{ki}+I_kh_{ij}).
\ee
This means that $F$ is a C-reducible metric.
\qed

\bigskip

There is an equal definition for Landsberg manifolds based on Finslerian connections. A Finsler manifold is called a Landsberg manifold if the Chern connection coincides with the Berwald connection.  With this definition of Landsberg manifolds,  Bejancu-Farran  introduce a new class of Finsler manifolds called generalized Landsberg manifolds. A  Finsler manifold  is said to be generalized Landsberg manifold if the $h$-curvature tensors of the Berwald and Chern connections coincide \cite{BF}. The relation between h-curvatures of Berwald and Chern connections is given by
\begin{equation}\label{GL1}
H^i_{jkl}=R^i_{\ jkl}+[L^i_{\ jl|k}-L^i_{\ jk|l}+L^i_{\ sk}L^s_{\ jl}-L^i_{\ sl}L^s_{\ jk}],
\end{equation}
where $H$ and $R$ are h-curvatures of Berwald and Chern connections respectively. Thus $F$ is a generalized
Landsberg metric if we have
\begin{equation}\label{GL2}
L^i_{\ jl|k}-L^i_{\ jk|l}+L^i_{\ sk}L^s_{\ jl}-L^i_{\ sl}L^s_{\ jk}=0.
\end{equation}
\begin{lem}\label{LemQ}
Let $(M, F)$ be a Finsler manifold. Then $F$ is a generalized Landsberg metric if  and only if the following equations hold
\begin{eqnarray}
&&L_{isk}L^s_{\ jl}-L_{isl}L^s_{\ jk}=0,\label{GL4} \\
&&L_{ijl|k}-L_{ijk|l}=0.\label{GL5}
\end{eqnarray}
\end{lem}
\begin{proof}
Fix $k$ and $l$ and  put
\[
Q_{ij}:=L_{ijl|k}-L_{ijk|l}+L_{isk}L^s_{\ jl}-L_{isl}L^s_{\ jk}.
\]
One can write
\[
Q_{ij}:=Q^s_{ij}+Q^a_{ij},
\]
where
\[
Q^s_{ij}:=\frac{1}{2}(Q_{ij}+Q_{ji}),\ \ \textrm{and} \ \ Q^a_{ij}:=\frac{1}{2}(Q_{ij}-Q_{ji}).
\]
It is easy to see  that $Q_{ij}=0$ if and only if $Q^s_{ij}=0$ and $Q^a_{ij}=0$. On the other hand, we have
\begin{eqnarray*}
Q_{ji}\!\!\!\!&=&\!\!\!\!\ L_{jil|k}-L_{jik|l}+L_{jsk}L^s_{\ il}-L_{\ jsl}L^s_{\ ik}\\
\!\!\!\!&=&\!\!\!\!\ L_{ijl|k}-L_{ijk|l}+L^s_{\ jk}L_{sil}-L^s_{\ jl}L_{sik}.
\end{eqnarray*}
Hence
\[
Q^s_{ij}=L_{ijl|k}-L_{ijk|l},
\]
and consequently
\[
Q^a_{ij}=L_{isk}L^s_{\ jl}-L_{isl}L^s_{\ jk}.
\]
This proves the Lemma.
\end{proof}

Therefore, every generalized Landsberg metric is a stretch metric. By the Theorem \ref{MainTheorem2}, we have the following.
\begin{cor}
Let $(M, F)$ be a generalized  P-reducible Finsler manifold with dimension $n\geq 3$. Suppose that $F$ be a  generalized Landsberg  metric. Then $F$ is  a  Randers metric.
\end{cor}

\noindent
Esmaeil Peyghan\\
Faculty  of Science, Department of Mathematics\\
Arak University\\
Arak. Iran\\
Email: epeyghan@gmail.com

\bigskip

\noindent
Akbar Tayebi\\
Faculty  of Science, Department of Mathematics\\
Qom University\\
Qom, Iran\\
Email:\ akbar.tayebi@gmail.com.

\bigskip

\noindent
Abbas Heydari\\
Faculty of Science, Department of Mathematics\\
Modares University\\
Tehran. Iran\\
Email: aheydari@modares.ac.ir.
\end{document}